\documentclass[11pt,letterpaper]{amsart}
\usepackage[margin= 1in]{geometry}
\usepackage[title]{appendix}
\usepackage{color,bbm,mathtools}
\usepackage{setspace}

\usepackage[english]{babel}
\usepackage{enumitem,stmaryrd}
\usepackage{booktabs}

\newtheorem{theorem}{Theorem}

\newtheorem{definition}[theorem]{Definition}
\newtheorem{proposition}[theorem]{Proposition}
\newtheorem{corollary}[theorem]{Corollary}
\newtheorem{lemma}[theorem]{Lemma}

\numberwithin{theorem}{section}
\numberwithin{equation}{section}

\begin{document}
 \title[]{Applications of optimal transport to Dyson Brownian Motions and beyond}
\author{Xuan Wu}
\address{Department of Mathematics, University of Illinois Urbana-Champaign, Altgeld Hall 
1409 W Green St, Urbana, IL 61801}
\email{xuanzw@illinois.edu}
\begin{abstract}
We develop a new method based on Caffarelli's contraction theorem in optimal transport to obtain sharp and uniform modulus of continuity estimates for $\beta$-Dyson Brownian motions with $\beta \geq 2$. Our method extends to a large class of random curve collections, which can be viewed as log-concave perturbations of Brownian motions, including the $\beta$-Dyson Brownian motion, the Air$\textup{y}_{\beta}$ line ensemble, the KPZ line ensemble, and the O'Connell-Yor line ensemble.
\end{abstract}
\maketitle

\section{Introduction}
For $\beta\geq 1$, the $\beta$-Dyson Brownian motion is the solution to the following system of stochastic differential equation (SDE)

\begin{align}\label{equ:DBM}
\displaystyle dX_j(t)=\frac{\beta}{2}  \sum_{i\neq j}\frac{dt}{X_j(t)-X_i(t)} +dB_j(t),\  1\leq j\leq N. 
\end{align}
We refer readers to \cite[Section 4.3]{MR2760897} for the existence and the uniqueness of the solution to \eqref{equ:DBM}. The $\beta$-Dyson Brownian motion is a canonical model for non-intersecting paths and is closely related to various important fields in probability, including the random matrix theory \cite{MR148397}, the KPZ universality \cite{MR1257246,MR3204148}, and the Schramm–Loewner evolution \cite{MR2004294}. 

In this paper, we develop a new method based on Caffarelli's contraction theorem in optimal transport to obtain sharp modulus of continuity estimates for $\beta$-Dyson Brownian motions. 

Our key observation is to view $\beta$-Dyson Brownian motions as log-concave perturbations of standard Brownian motions. To the best of our knowledge, this is a new perspective on the Dyson Brownian motions and has not been explored before.
Under this perspective, we are able to apply the techniques from optimal transport and to perform a direct comparison between $\beta$-Dyson Brownian motions and Brownian motions. Using this comparison, we establish sharp and explicit bounds in Theorem~\ref{thm:1}.

We write $\mathbb{W}^N=\{x\in\mathbb{R}^N\, :\, x_1>x_2>\dots>x_N\}$ for the Weyl chamber and $\overline{\mathbb{W}}^N$ for its closure. For a continuous function $z(t)$ defined on $t\in [a,b]$ and $\alpha\in (0,1)$, we write
\begin{align*}
\|z\|_{\alpha,[a,b]}=\sup_{t,s\in [a,b], t\neq s}\frac{|z(t)-z(s)|}{|t-s|^\alpha}.
\end{align*}

\begin{theorem}\label{thm:1}
Fix $\beta\geq 2$. Let $(X_1(t),X_2(t),\dots,X_N(t))$ be a $\beta$-Dyson Brownian motion starting at $x_*\in\overline{\mathbb{W}}^N$. We write $\bar{X}_j(t)$ for the mean of $X_j(t)$ and denote $\hat{X}_j(t)=X_j(t)-\bar{X}_j(t)$. Then the following statements hold.
\begin{enumerate}[label=(\roman*)]
  \item For all $p\geq 1$, $1\leq j\leq N$ and $t,s\in [0,\infty)$, we have
\begin{align*}
\mathbb{E}|\hat{X}_j(t)-\hat{X}_j(s)|^p\leq N_p|t-s|^{p/2},
\end{align*}
where $N_p=\mathbb{E}|\mathcal{N}|^p$ for a standard normal $\mathcal{N}$.
  \item There exist universal constants $C_1,C_2>0$ such that for all $0\leq a<b$, $1\leq j\leq N$ and $K\geq 0$, we have
\begin{align*} 
\mathbb{P}\left( \sup_{t,s\in [a,b], t\neq s}  \frac{|\hat{X}_j(t)-\hat{X}_j(s)|}{\sqrt{|t-s| \log (2(b-a)/|t-s|)}} > K \right)\leq C_1e^{-C_2K^2}.
\end{align*}
  \item For any $p>0$ and $\alpha\in (0,\frac{1}{2})$, there exists a constant $C(\alpha,p)>0$ such that for all $0\leq a<b$ and $1\leq j\leq N$, we have
  \begin{align*}
  \mathbb{E} \|\hat{X}_j\|_{\alpha,[a,b]}^p  \leq  C(\alpha,p)(b-a)^{p/2-\alpha p }. 
\end{align*}   
\end{enumerate}
\end{theorem}

Note that the right hand side upper bounds in the above theorem are uniform for every layer and are exactly the same bounds for a standard Brownian motion. {The modulus of continuity estimates for $\beta$-Dyson Brownian motions have been previously studied in \cite{landon2020edgescalinglimitdyson,huang2024convergenceframeworkairybetaline,gorin2024}  , but none of them are as sharp as those in Theorem~\ref{thm:1}.}

Besides $\beta$-Dyson Brownian motions, there are also many other interesting examples that our approach also applies, as they all share the log-concave feature. We develop a general approach applicable to any random continuous functions which can be viewed as log-concave perturbations of either Brownian motions or Brownian bridges. We denote by $\mathbf{LC}$ the collection of all these log-concave perturbations together with their distributional limits. Theorem~\ref{thm:1} holds for all members in $\mathbf{LC}$, including the following interesting examples: the Air$\textup{y}_\beta$ line ensemble, the O'Connell-Yor line ensemble \cite{MR2952082}, the KPZ line ensemble \cite{CH16,MR3439221, Wu21tightness}, the Bessel line ensemble \cite{MR4605337} and Airy line ensembles with wanders \cite{MR2642890,dimitrov2024}. We refer readers to Section~\ref{sec:2} for the detail definition of $\mathbf{LC}$ and the precise statement of the general case in Theorem~\ref{thm:2}.

\subsection{Motivation and Applications} We discuss motivation and record four applications in this subsection. Our investigation is primarily motivated by the construction of the directed landscape from Brownian last passage percolation by Dauvergne, Ortmann and Vir{\'a}g \cite{Dauvergne_2022}. A crucial ingredient in \cite{Dauvergne_2022} is a modulus of continuity estimate, established by Dauvergne and Vir{\'a}g in a separate paper \cite{MR4260467}, for the Airy line ensemble, the edge scaling limit of the $\beta$-Dyson Brownian motion with $\beta=2$, which was first described by Pr{\"a}hofer and Spohn \cite{MR1933446} (see also \cite{MR3152753,dauvergne2019uniform,MR4260467,MR4317710,Wu21tightness,aggarwal2023}
).

Let's first introduce the Air$\textup{y}_\beta$ line ensemble. Consider $\beta$-Dyson Brownian motions under the edge scaling for $\beta>0$. For $\beta>0$, let $X^{\beta,N}(t)$ be the $N$-dimensional $\beta$-Dyson Brownian motion starting at the origin. We perform the edge scaling to $X^{\beta,N}$ at $t=\frac{2}{\beta}$ by defining $\mathcal{A}^{\beta,N}(t)$ through
\begin{align}\label{equ:Airybeta}
\mathcal{A}^{\beta,N}_j(t)=N^{1/6} X^{\beta,N}_j\left(\frac{2}{\beta}  \left(1+\frac{t}{N^{1/3}} \right) \right) -2N^{2/3}-tN^{1/3}+\frac{t^2}{4},\ 1\leq j\leq N. 
\end{align}

When $N$ goes to infinity, $\mathcal{A}^{\beta,N}$ converges to $\mathcal{A}^\beta$, referred to as the Air$\textup{y}_{\beta}$ line ensemble \cite{huang2024convergenceframeworkairybetaline,landon2020edgescalinglimitdyson,gorin2024} .  
See also \cite{MR3568040} for the bulk limit.  


In the following corollary, we provide an alternative short proof for the tightness of $\mathcal{A}^{\beta,N}$ for $\beta\geq 2$.
\begin{corollary}\label{prop:4}
Fix $\beta\geq 2$. For any compact set $K\subset \mathbb{R}\times\mathbb{N}$, the marginal law of $\mathcal{A}^{\beta,N}$ on $C(K,\mathbb{R})$ is tight.
\begin{proof}
Let $m^{\beta,N}_j(t)$ be the expectation of $\mathcal{A}^{\beta,N}_j(t)$ and set 
$$\hat{\mathcal{A}}^{\beta,N}_j(t)=\sqrt{\beta/2} \big(\mathcal{A}^{\beta,N}_j(t)- m^{\beta,N}_j(t)+m^{\beta,N}_j(0)\big).$$ 
At $t=0$, $\sqrt{2/\beta}\hat{\mathcal{A}}^{\beta,N}(0)=\mathcal{A}^{\beta,N}(0)$ is the edge scaling of the Gaussian $\beta$ ensemble and its convergence has been established in \cite{MR2813333}. From Theorem~\ref{thm:1} (ii), we have 
\begin{align*}
\mathbb{P}\left( \sup_{t,s\in [a,b], t\neq s}   |\hat{\mathcal{A}}^{\beta,N}_j(t)-\hat{\mathcal{A}}^{\beta,N}_j(s)|  >K|b-a|^{1/2} \right)\leq C_1e^{-C_2K^2}.
\end{align*}
This allows us to apply \cite[Theorem 7.3]{Bil} to ensure the marginal law of $\hat{\mathcal{A}}^{\beta,N}$ on $C(K,\mathbb{R})$ is tight. From the Brownian scaling and the convergence of $\mathcal{A}^{\beta,N}(0)$, $m^{\beta,N}_j(t)-m^{\beta,N}_j(0)$ converges locally uniformly to zero. Hence we conclude $\mathcal{A}^{\beta,N}$ is tight on $C(K,\mathbb{R})$.
\end{proof}
\end{corollary}

Motivated by the construction of the directed landscape from the Brownian last passage percolation by Dauvergne, Ortmann and Vir{\'a}g \cite{Dauvergne_2022}, we are interested in estimates on modulus of continuity for Air$\textup{y}_\beta$ line ensembles for all layers. Unexpectedly, we are able to obtain the following uniform estimates for all layers and the estimates are exactly the same ones that can be obtained for a standard Brownian motion.

\begin{corollary}\label{prop:3}
There exist universal constants $C_1, C_2>0$ such that for any $\beta\geq 2$, $a<b$, $j\in\mathbb{N}$, and $K\geq 0$, it holds that
\begin{align}\label{equ:Airybeta2}
\mathbb{P}\left( \sup_{t,s\in [a,b], t\neq s}  \frac{|\mathcal{A}^{\beta}_j(t)-\mathcal{A}^{\beta}_j(s)|}{\sqrt{|t-s| \log (2(b-a)/|t-s|)}} >\sqrt{\frac{2}{\beta}} K \right)\leq C_1e^{-C_2K^2}.
\end{align}  
\end{corollary}
For $\beta=2$, the bound \eqref{equ:Airybeta2} was previously obtained by Dauvergne and Vir{\'a}g in \cite[Theorem 1.5]{MR4260467} with the constant $C_1=O(e^j)$, which plays an important role in the construction of the directed landscape in \cite{Dauvergne_2022}.

Our next application concerns KPZ line ensembles \cite{CH16,MR3439221}, an important model closely related to the KPZ equation \cite{KPZ}. The KPZ line ensemble plays a role for the KPZ equation similar to the one the Airy line ensemble plays for the directed landscape. It has been established that the directed landscape can be constructed by the Airy line ensemble through a deterministic map \cite{Dauvergne_2022,dauvergne2022scaling}. Analogously, it is conjectured that the fundamental solution to the KPZ equation can be constructed from the KPZ line ensemble \cite{wu2023kpzequationdirectedlandscape}. Partial progress on this conjecture has been made in \cite{wu2023kpzequationdirectedlandscape}, but a key missing component is the modulus of continuity estimate for the KPZ line ensemble. As an application of Theorem~\ref{thm:2}, we fill this gap in the corollary below.

\begin{corollary}\label{prop:5}
Fix $T>0$. Let $\mathcal{X}^T$ be the KPZ line ensemble with parameter $T$ and set $\hat{\mathcal{X}}^{T}_j(t)=\mathcal{X}^{T}_j(t)+2^{-1}t^2$. Then for any $p>0$ and $\alpha\in (0,\frac{1}{2})$, there exists a constant $C(\alpha,p)>0$ such that for all $a<b$ and $j\in \mathbb{N} $, we have  
\begin{align}\label{equ:KPZbound}
\mathbb{E} \|\hat{\mathcal{X}}^{T}_j\|_{\alpha,[a,b]}^p  \leq  C(\alpha,p)(b-a)^{p/2-\alpha p }. 
\end{align}   
\end{corollary}

Corollary~\ref{prop:5} can be used to confirm the Conjecture 1.12 in \cite{wu2023kpzequationdirectedlandscape}.

{\color{black}
The following corollary concerns the log-concavity of finite-dimensional marginal distributions for the $\beta$-Dyson Brownian motion, the Airy$_\beta$ line ensemble, both with $\beta \geq 2$, the O'Connell-Yor line ensemble, and the KPZ line ensemble.

\begin{corollary}\label{cor:finite_dim_logconcave}
The finite-dimensional marginal distributions for the $\beta$-Dyson Brownian motion, the Airy$_\beta$ line ensemble, both with $\beta \geq 2$, the O'Connell-Yor line ensemble, and the KPZ line ensemble are log-concave.
\end{corollary}

The log-concavity for the Airy line ensemble was recently proved in \cite{baslingker}. In Corollary~\ref{cor:arginal_logconcave}, we prove that finite-dimensional marginal distributions for any line ensembles in $\mathbf{LC}$ are log-concave. We show that the $\beta$-Dyson Brownian motion with $\beta\geq 2$ and the O'Connell-Yor line ensemble are in $\mathbf{LC}$ in Propositions~\ref{prop:DBM} and \ref{prop:OY} respectively. The Airy$_\beta$ line ensembles and the KPZ line ensemble are the scaling limits of the $\beta$-Dyson Brownian motions \cite{landon2020edgescalinglimitdyson,huang2024convergenceframeworkairybetaline,gorin2024} and the O'Connell-Yor line ensemble \cite{CH16,Nic21} respectively. Since $\mathbf{LC}$ is closed under convergence in distribution, the Airy$_\beta$ line ensemble with $\beta\geq 2$ and the KPZ line ensemble are also in the class $\mathbf{LC}$.
}

\subsection{Idea of Proofs}\label{subsec:1}
In this subsection, we explain the proof ideas. We begin with explaining the concept of a log-concave perturbation of Brownian motions, with $\beta$-Dyson Brownian motions as the primary example.

The $\beta$-Dyson Brownian motions can be constructed through the Girsanov theorem as follows. Let $X(t)=(X_1(t),X_2(t),\dots,X_N(t))$ be independent Brownian motions starting at $x_*\in\mathbb{W}^N$ defined on a filtered probability space $(\Omega,\mathcal{F}_t, \mathbb{P}_0)$. For $\beta\geq 1$, there is a non-negative martingale $M_\beta(t)$ such that $X(t)$ is a $\beta$-Dyson Brownian motion with respect to $M_\beta\cdot\mathbb{P}_0$, the probability measure obtained after tilting $\mathbb{P}_0$ by $M_\beta(t)$. The martingale $M_\beta(t)$ takes the following form. If $X(s)\notin\mathbb{W}^N\ \textup{for some}\ s\in [0,t]$, then $M_\beta(t)=0$. If $X(s)\in\mathbb{W}^N\ \textup{for all}\ s\in [0,t]$, then
\begin{align}\label{equ:Mbeta}
M_\beta(t)= \frac{h_\beta(X(t))}{h_\beta(X(0))}\exp\left(-\int_0^t V_\beta(X(s))\, ds\right),
\end{align}
where $h_\beta$ and $V_\beta$ are given by
\begin{align}\label{def:h_beta}
h_\beta(x)=\left\{\begin{array}{cc}
\prod_{1\leq i<j\leq N}|x_i-x_j|^{\beta/2}, & x\in\mathbb{W}^N,\\
0, & x\notin \mathbb{W}^N,
\end{array} \right.
\end{align}
and
\begin{align}\label{def:V_beta}
V_\beta(x)=\left\{ \begin{array}{cc}
\frac{\beta(\beta-2)}{4}\sum_{1\leq i<j\leq N}\frac{1}{(x_i-x_j)^2}, & x\in\mathbb{W}^N,\\
\infty, & x\notin \mathbb{W}^N.
\end{array} \right.
\end{align}
When $\beta\geq 2$, it can be checked that $V_\beta$ and $-\log h_\beta$ are convex functions on $\mathbb{W}^N$. Therefore, $\beta$-Dyson Brownian motions can be viewed as log-concave perturbations of independent Brownian motions. Another interesting example of a log-concave perturbation of Brownian motion is the diffusion process given by
\begin{align}\label{equ:OC}
dX_j(t)= \nabla_j\log  \psi_0 (X(t))dt+dB_j(t),\ 1\leq j\leq N,
\end{align}
where $\psi_0$ is the $\mathfrak{gl}_{N}$-Whittaker function which solves
$$ \frac{1}{2}\Delta\psi_0(x)=\left(\sum_{j=1}^{N-1}e^{x_{j+1}-x_j}\right)\psi_0(x). $$
The SDE \eqref{equ:OC} appears in the context of O'Connell-Yor directed polymer \cite{MR1865759} and is related to the quantum Toda lattice \cite{MR2952082}. Similar to \eqref{equ:DBM}, \eqref{equ:OC} can be solved through the Girsanov theorem. The associated martingale is similar to $M_\beta(t)$ in \eqref{equ:Mbeta}, with $h_\beta$ replaced by $\psi_0$ and $V_\beta$ replaced by $\sum_{j=1}^{N-1}e^{x_{j+1}-x_j}$. In general, for a positive function $u(x)$ with $-\log u(x)$ and $\Delta u(x)/(2u(x))$ both being convex, a solution to the SDE
\begin{align}\label{equ:u}
dX_j(t)=\nabla_j\log  u (X(t))dt+dB_j(t),\ 1\leq j\leq N
\end{align}
can be viewed as a log-concave perturbation of Brownian motions provided the associated local martingale is a true martingale.


Next, we explain the idea of proving Theorem~\ref{thm:2}. Let $\hat{\gamma}$ be a finite-dimensional Gaussian measure on $\mathbb{R}^d$ and $\hat{\mu}$ be a log-concave perturbation of $\hat{\gamma}$. In other words, $\hat{\mu}$ is absolutely continuous with respect to $\hat{\gamma}$ and $d\hat{\mu}/d\hat{\gamma}$ is log-concave. Caffarelli’s contraction theorem \cite[Theorem 11]{MR1800860} says that the Brenier map $T:\mathbb{R}^d\to\mathbb{R}^d$ which pushes forward $\hat{\gamma}$ onto $\hat{\mu}$ is a contraction. In other words, $|T(w_1)-T(w_2)|\leq |w_1-w_2|$ for all $w_1,w_2\in \mathbb{R}^d$. This directly implies that for $g(w)=|w|^2$, one has
\begin{equation}\label{equ:comparison}
\mathbb{E}_{\hat{\mu}}[g(w-\bar{w}_{\mu})]\leq \mathbb{E}_{\hat{\gamma}}[g(w-\bar{w}_{\gamma})],
\end{equation}  
where $\bar{w}_{\mu}$ and $\bar{w}_{\gamma}$ are the mean values of $w$ with respect to $\hat{\mu}$ and $\hat{\gamma}$ respectively. Harg{\'e} extended the estimate \eqref{equ:comparison} and showed that it holds true for any convex function $g$ \cite[Theorem 1.1]{MR2095937}. In our setting, the law of independent Brownian motions/bridges is a Gaussian measure on an infinite-dimensional space. Even though \eqref{equ:comparison} does not apply to the infinite-dimensional setting, our crucial observation is that log-concavity is preserved in taking a finite-dimensional marginal. Based on this observation, we establish a direct comparison between Brownian motions/bridges and their log-concave perturbations. We note that optimal transport maps on Wiener spaces have been constructed \cite{MR2036490} and some of the contraction properties have been studied in \cite{MR4797372}. 

\subsection*{Notation}
For integers $k,\ell\in\mathbb{Z}$ with $k\leq \ell$, we denote $\llbracket k, \ell \rrbracket=\mathbb{Z}\cap [k,\ell]$. We write $\mathbb{R}^{\llbracket k, \ell \rrbracket}$ for the collection $\{x=(x_k,x_{k+1},\dots,x_{\ell})\, |\, x_j\in\mathbb{R},\ j\in\llbracket k, \ell \rrbracket\}$. We use $C(\llbracket k, \ell \rrbracket\times [a,b],\mathbb{R})$ to denote the collection of continuous functions defined on $\llbracket k, \ell \rrbracket\times [a,b]$. For $z\in C(\llbracket k, \ell \rrbracket\times [a,b],\mathbb{R})$, $j\in\llbracket k,\ell \rrbracket$, and $t\in [a,b]$, we often write $z_j(t)=z(j,t)$ and $z(t)=(z_k(t),z_{k+1}(t),\dots, z_{\ell}(t))\in\mathbb{R}^{\llbracket k, \ell \rrbracket}$. We use $C_0(\llbracket k, \ell \rrbracket\times [a,b],\mathbb{R})$ to denote the subset of $C(\llbracket k, \ell \rrbracket\times [a,b],\mathbb{R})$ in which $z(a)=0\in\mathbb{R}^{\llbracket k, \ell \rrbracket}$. We similarly use $C_{0,0}(\llbracket k, \ell \rrbracket\times [a,b],\mathbb{R})$ to denote the subset of $C(\llbracket k, \ell \rrbracket\times [a,b],\mathbb{R})$ in which $z(a)=z(b)=0\in\mathbb{R}^{\llbracket k, \ell \rrbracket}$. 

\subsection*{Acknowledgment} The author thanks Promit Ghosal for many helpful discussions and thanks Erik Bates, Evgeni Dimitrov, Li-Cheng Tsai and Lingfu Zhang for many valuable comments on an early draft of the paper. The author is partially supported by the NSF through NSF-2348188 and by the Simons Foundation through MPS-TSM-00007939.

\section{The class $\mathbf{LC}$ and the main result}\label{sec:2}
In this section, we define the class $\mathbf{LC}$ and prove the main result, Theorem~\ref{thm:2}. The class $\mathbf{LC}$ can be interpreted as a collection of random continuous functions obtained by reweighting Brownian motions or Brownian bridges by a log-concave weight. We begin by introducing the notation for the law of Brownian motions and Brownian bridges. We denote $\mathsf{s}=(R,x,\eta(t))$ as a triple that consists of the data:

\begin{itemize}
\item $R=\llbracket k,\ell \rrbracket\times [a,b]\subset \mathbb{Z}\times\mathbb{R}$,
\item a vector $x=(x_k,x_{k+1},\dots,x_\ell)$ in $\mathbb{R}^{\llbracket k,\ell \rrbracket}$,
\item a continuous function $\eta\in C_0(\llbracket k,\ell  \rrbracket\times [a,b],\mathbb{R})$.
\end{itemize}

For a triple $\mathsf{s}=(R,x,\eta(t))$, we define the shifted Wiener measure $\gamma_{\mathsf{s}}$ as follows. Let $\{B_j(t)\}_{j\in \llbracket k,\ell \rrbracket}$ be independent Brownian motions on $t\in [a,b]$ with diffusion constant $1$ and $B_j(a)=0$. We denote by $\gamma_{\mathsf{s}}$ the law of $\{B_j(t)+x_j+\eta_j(t)\}_{j\in \llbracket k,\ell \rrbracket}$, and view $\gamma_{\mathsf{s}}$ as a Borel probability measure on $C(\llbracket k,\ell  \rrbracket\times [a,b],\mathbb{R})$. For fixed $R=\llbracket k, \ell \rrbracket\times [a,b]$, we denote by $\mathsf{S}^{R}$ the collection of all triples $\mathsf{s}$ of the form $\mathsf{s}=(R,x,\eta(t))$. The bridge counterpart is defined similarly. Denote $\breve{\mathsf{s}}=(R,x,y,\breve{\eta}(t))$ as a tuple that consists of the data:

\begin{itemize}
\item $R=\llbracket k,\ell \rrbracket\times [a,b]\subset \mathbb{Z}\times\mathbb{R}$,
\item two vectors $x=(x_k,x_{k+1},\dots,x_\ell)$ and $y=(y_k,y_{k+1},\dots,y_\ell)$ in $\mathbb{R}^{\llbracket k,\ell \rrbracket}$,
\item a continuous function $\breve{\eta}\in C_{0,0}(\llbracket k,\ell  \rrbracket\times [a,b],\mathbb{R})$.
\end{itemize}

Let $\{\breve{B}_j(t)\}_{j\in \llbracket k,\ell \rrbracket}$ be independent Brownian bridges on $t\in [a,b]$ with $\breve{B}_j(a)=\breve{B}_j(b)=0$. We denote by $\gamma_{\breve{\mathsf{s}}}$ the law of $\{\breve{B}_j(t)+\frac{(b-t)x_j+(t-a)y_j}{b-a} +\breve{\eta}_j(t)\}_{j\in \llbracket k,\ell \rrbracket}$, and also view $\gamma_{\breve{\mathsf{s}}}$ as a Borel probability measure on $C(\llbracket k,\ell  \rrbracket\times [a,b],\mathbb{R})$. For fixed $R=\llbracket k, \ell \rrbracket\times [a,b]$, we denote by $\breve{\mathsf{S}}^{R}$ the collection of all tuples $\breve{\mathsf{s}}$ of the form $\breve{\mathsf{s}}=(R,x,y,\breve{\eta}(t))$.

Next, we introduce log-concave weights of the form $e^{-\mathbf{H}(z)}$ for a convex Hamiltonian $\mathbf{H}$. Fix $R=\llbracket k,\ell \rrbracket\times [a,b]$ and consider the following data:

\begin{itemize}
\item $t_1,t_2,\dots,t_m\in [a,b]$ for some $m\in\mathbb{N}_0$,
\item convex functions $f_i:\mathbb{R}^{\llbracket k,\ell \rrbracket}\to \mathbb{R}$ for $1\leq i\leq m$,
\item a continuous function $F(t,x)$ for $(t,x)\in [a,b]\times\mathbb{R}^{\llbracket k,\ell \rrbracket}$ which is convex in $\mathbb{R}^{\llbracket k,\ell \rrbracket}$ for each fixed $t$.
\end{itemize}

The corresponding Hamiltonian is a convex functional on $C(\llbracket k,\ell \rrbracket\times [a,b],\mathbb{R})$ given by
\begin{align}\label{equ:H}
\mathbf{H}(z)= \sum_{i=1}^m f_i(z(t_i))+\int_a^b F(t,z(t))\, dt.
\end{align}
We write $\mathsf{H}^R$ for the collection of all Hamiltonians of the form \eqref{equ:H} in which $f_i$ and $F$ satisfy the convexity assumptions above. We use $\mathsf{H}^R_+$ to denote the subset of $\mathsf{H}^R$ in which $\mathbf{H}(z)\geq 0$ for all $z\in C(\llbracket k,\ell \rrbracket\times [a,b],\mathbb{R})$. 

We consider probability measures proportional to $e^{-\mathbf{H}(z)}\gamma_{\mathsf{s}}$ or $e^{-\mathbf{H}(z)}\gamma_{\breve{\mathsf{s}}}$. To ensure these are well-defined, we show the finiteness of the total integrals of $e^{-\mathbf{H}(z)}\gamma_{\mathsf{s}}$ and $e^{-\mathbf{H}(z)}\gamma_{\breve{\mathsf{s}}}$ in the following two lemmas.

\begin{lemma}\label{lem:conti_supp}
Let $F(t,x)$ be a continuous function for $(t,x)\in [a,b]\times\mathbb{R}^{\llbracket k,\ell \rrbracket}$ which is convex in $\mathbb{R}^{\llbracket k,\ell \rrbracket}$ for each fixed $t$. Then there exist a continuous function $A:[a,b]\to\mathbb{R}^{\llbracket k,\ell \rrbracket}$ and $c_0\in\mathbb{R}$ such that $F(t,x)\geq \left\langle A(t),x \right\rangle+c_0$ for all $(t,x)\in [a,b]\times \mathbb{R}^{\llbracket k,\ell \rrbracket}$.
\begin{proof}
Let $c_0=\min_{t\in [a,b]} F(t,0)-1$. For each $t\in [a,b]$, the convexity of $F(t,\cdot)$ implies there exists a unique $x(t)\in\mathbb{R}^{\llbracket k,\ell \rrbracket}$ that minimizes $|x|^2+(F(t,x)-c_0)^2$. The uniqueness of $x(t)$ and the continuity of $F$ ensure $x(t)$ is continuous in $t\in [a,b]$. The convexity of $F(t,\cdot)$ and the minimizing property of $x(t)$ then imply $F(t,x)\geq \left\langle A(t),x \right\rangle+c_0$ for $A(t)=\frac{-x(t)}{F(t,x(t))-c_0}$.
\end{proof}
\end{lemma}


\begin{lemma}\label{lem:mu0}
Fix $R=\llbracket k,\ell \rrbracket\times [a,b]\subset \mathbb{Z}\times \mathbb{R}$, $\eta\in C_0(\llbracket k,\ell \rrbracket\times [a,b],\mathbb{R})$, and $\mathbf{H}\in\mathsf{H}^{R}$. There exists $\eta_*(t)\in C_0(\llbracket k,\ell \rrbracket\times [a,b],\mathbb{R})$ such that the following holds. For any $\mathsf{s}(x)=(R,x,\eta(t))\in\mathsf{S}^{R}$, we have 
\begin{equation}\label{equ:eHbdd_0}
e^{-\mathbf{H}(z)}\gamma_{\mathsf{s}(x)}=Z(x)e^{- \mathbf{H}_+(z)}\gamma_{\mathsf{s}_+(x)},
\end{equation}
where $\mathsf{s}_+(x)=(R,x,\eta(t)+\eta_*(t))$, $\mathbf{H}_+\in\mathsf{H}^R_+$, and $Z(x)\in (0,\infty)$. Moreover, $\log Z(x)$ is an affine function in $x$.


Similarly, fix $\breve{\eta}(t)\in C_{0,0}(\llbracket k,\ell \rrbracket\times [a,b],\mathbb{R})$. There exists $\breve{\eta}_*(t)\in C_{0,0}(\llbracket k,\ell \rrbracket\times [a,b],\mathbb{R})$ such that the following holds. For any $\breve{\mathsf{s}}(x,y)=(R,x,y,\breve{\eta}(t))\in\breve{\mathsf{S}}^{R}$, we have 
\begin{equation}\label{equ:eHbdd2_0}
e^{-\mathbf{H}(z)}\gamma_{\breve{\mathsf{s}}(x,y)}=\breve{Z}(x,y)e^{- \mathbf{H}_+(z)}\gamma_{\breve{\mathsf{s}}_+(x,y)},
\end{equation}
where $\breve{\mathsf{s}}_+(x,y)=(R,x,y,\breve{\eta}(t)+\breve{\eta}_*(t))$, $\mathbf{H}_+\in\mathsf{H}^R_+$, and $\breve{Z}(x,y)\in (0,\infty)$. Moreover, $\log \breve{Z}(x,y)$ is an affine function in $(x,y)$. 

\begin{proof}
Let $\mathbf{H}$ be given by the form \eqref{equ:H}. By the convexity of $f_i$, $f_i(x)\geq \left\langle y_i,x \right\rangle+c_i$ for some $y_i\in\mathbb{R}^{\llbracket k,\ell \rrbracket }$ and $c_i\in\mathbb{R}$. From Lemma~\ref{lem:conti_supp}, $F(t,x)\geq \left\langle A(t),x \right\rangle+c_0$ for some continuous vector $A(t)$ and $c_0\in\mathbb{R}$.  Let 
$$\mathring{\mathbf{H}}(z)= \sum_{i=1}^m\left(  \left\langle y_i, z(t_i)\right\rangle+c_i\right) +\int_a^b \left\langle A(t),z(t)\right\rangle +c_0 \, dt.$$
Then ${\mathbf{H}}_+= \mathbf{H}-\mathring{\mathbf{H}} $ is convex and non-negative. Therefore ${\mathbf{H}}_+\in\mathsf{H}^R_+$. The Cameron-Martin theory \cite[Chapter 8]{MR1474726} implies there exists $\eta_*(t)\in C_0(\llbracket k,\ell \rrbracket\times [a,b],\mathbb{R})$ such that for any $\mathsf{s}(x)=(R,x,\eta(t))\in\mathsf{S}^R$, we have
$$ \gamma_{  \mathsf{s}_+(x)} = \frac{1}{Z(x)} e^{-\mathring{\mathbf{H}}(z)}\gamma_{\mathsf{s}(x)},$$
where $\mathsf{s}_+(x)=(R,x,\eta(t)+\eta_*(t))$ and $Z(x)\in (0,\infty)$. Then \eqref{equ:eHbdd_0} follows by rearranging terms. To show $\log Z(x)$ is an affine function, we write $\mathbb{E}_x$ for the expectation with respect to $\gamma_{\mathsf{s}(x)}$. Then
\begin{align*}
Z(x)=\mathbb{E}_x[e^{-\mathring{\mathbf{H}}(z)}]=\mathbb{E}_0[e^{-\mathring{\mathbf{H}}(z+x)}]=\mathbb{E}_0[e^{-\mathring{\mathbf{H}}(z)}]\cdot \exp\left( -\sum_{i=1}^m \left\langle y_i,x \right\rangle-\int_a^b \left\langle A(t),x \right\rangle\, dt \right).
\end{align*}
This ensures $\log Z(x)$ is an affine function in $x$. The proof of the bridge counterpart is similar, and we omit the details.

\end{proof} 
\end{lemma}

From \eqref{equ:eHbdd_0} and \eqref{equ:eHbdd2_0},  the total integrals of $e^{-\mathbf{H}(z)}\gamma_{\mathsf{s}}$ and $e^{-\mathbf{H}(z)}\gamma_{\breve{\mathsf{s}}}$ are finite. Hence we can normalize $e^{-\mathbf{H}(z)}\gamma_{\mathsf{s}}$ and $e^{-\mathbf{H}(z)}\gamma_{\breve{\mathsf{s}}}$ to obtain probability measures.

\begin{definition}\label{def:mu}
Fix $R=\llbracket k,\ell \rrbracket\times [a,b]\subset \mathbb{Z}\times \mathbb{R}$. For any $\mathsf{s}\in\mathsf{S}^R$ and $\mathbf{H}\in\mathsf{H}^{R}$, we write $\mu_{\mathsf{s},\mathbf{H}}$ for the probability measure on $C(\llbracket k,\ell \rrbracket\times [a,b],\mathbb{R})$ that is proportional to $e^{-\mathbf{H}(z)}\gamma_{\mathsf{s}}$. Similarly, for any $\breve{\mathsf{s}}\in\breve{\mathsf{S}}^R$ and $\mathbf{H}\in\mathsf{H}^{R}$, we write $\mu_{\breve{\mathsf{s}},\mathbf{H}}$ for the probability measure that is proportional to $e^{-\mathbf{H}(z)}\gamma_{\breve{\mathsf{s}}}$. 
\end{definition}

{\color{black}The next lemma shows the log-concavity of the Radon-Nikodym derivative between finite-dimensional marginals of $\mu_{\mathsf{s},\mathbf{H}}$ and $\gamma_\mathsf{s}$.

\begin{lemma}\label{lem:main}
Fix $R=\llbracket k,\ell \rrbracket\times [a,b]\subset \mathbb{Z}\times \mathbb{R}$, $\mathsf{s}=(R,x,\eta(t))\in\mathsf{S}^{R}$ and $\mathbf{H}\in\mathsf{H}^{R}$. Let $\Pi\subset (a,b]$ be a finite set. Let $\hat{\gamma}$ and $\hat{\mu}$ be the probability measures on $\mathbb{R}^{\llbracket k,\ell \rrbracket\times\Pi }$ which are marginal laws of $\gamma_{\mathsf{s}}$ and $\mu_{\mathsf{s},\mathbf{H}}$ on $\{ z_j(t) \}_{ (j,t)\in\llbracket k,\ell \rrbracket \times\Pi}$  respectively. Then $\frac{d\hat{\mu}}{d\hat{\gamma}}$ is log-concave. When $\Pi\subset (a,b)$, the same holds for the bridge counterpart $\gamma_{\breve{\mathsf{s}}}$ and $\mu_{\breve{\mathsf{s}},\mathbf{H}}$.
\end{lemma}

The proof of Lemma~\ref{lem:main} is postponed to the end of this section.} In the following, we define the notion of convergence for probability measures and random continuous functions.

\begin{definition}\label{def:locallyweak}
Let $\Lambda_{i}\subset\mathbb{Z}$, $I_i\subset\mathbb{R}$ be (possibly infinite) intervals for $i\in\mathbb{N}\cup\{\infty\} $. Let $\mu_i$, $i\in\mathbb{N}\cup\{\infty\}$, be Borel probability measures on $C(\Lambda_i\times I_i,\mathbb{R})$. We say $\mu_i$ converges locally weakly to $\mu_\infty$ if the following holds. First, $\Lambda_{i}\times I_i\subset\Lambda_{\infty}\times I_{\infty}$ for all $i\in\mathbb{N}$. Second, for any compact subset $K\subset \Lambda_\infty\times I_\infty$, $K\subset \Lambda_i\times I_i$ for $i$ large enough. Third, the marginal law of $\mu_i$ on $C(K,\mathbb{R})$ converges weakly to the one of $\mu_\infty$ on $C(K,\mathbb{R})$.

Let $\mathcal{L}^{(i)}$, $i\in\mathbb{N}\cup\{\infty\}$ be a sequence of random continuous functions defined on $\Lambda_{i}\times I_i$. We say	$\mathcal{L}^{(i)}$ converges in distribution to $\mathcal{L}^{(\infty)}$ if the law of $\mathcal{L}^{(i)}$ locally weakly converges to the one of $\mathcal{L}^{(\infty)}$ 
\end{definition}

We are ready to define the class $\mathbf{LC}$.

\begin{definition}
Let $\mathring{\mathfrak{M}}$ be the collection of measures of the form $\mu_{\mathsf{s},\mathbf{H}}$ or $\mu_{\breve{\mathsf{s}},\mathbf{H}}$ given in Definition~\ref{def:mu}. Let $\mathfrak{M}$ be the collection of probability measures  that are the locally weak limit (in the sense of Definition~\ref{def:locallyweak}) of sequences in $\mathring{\mathfrak{M}}$. Let $\Lambda\subset\mathbb{Z}$, $I\subset\mathbb{R}$ be (possibly infinite) intervals. A random continuous function $\mathcal{L}=\{\mathcal{L}(j,t),\ j\in\Lambda, t\in I\}$ is said to be in the class $\mathbf{LC}$ if its law belongs to $\mathfrak{M}$.
\end{definition}

{\color{black}From the above definition, the class $\mathbf{LC}$ is closed under convergence in distribution. Consequently, we have the following corollary:

\begin{corollary}\label{cor:arginal_logconcave}
If $\mathcal{L}$ is in the class $\mathbf{LC}$, then any finite-dimensional marginal distribution of $\mathcal{L}$ is log-concave.

\begin{proof}
From Lemma~\ref{lem:main} and the Pr{\'e}kopa–Leindler Theorem \cite[Theorem 	13.9]{MR2814377}, finite-dimensional marginal distributions of $\mu_{\mathsf{s},\mathbf{H}}$ and $\mu_{\breve{\mathsf{s}},\mathbf{H}}$ are log-concave. Since log-concavity is preserved under convergence in distribution \cite[Theorem 2.10]{MR954608}, the assertion follows.
\end{proof}
\end{corollary}
}

For all members in $\mathbf{LC}$, we have the following estimates, uniform for every layer.

\begin{theorem}\label{thm:2}
Let $\Lambda \subset\mathbb{Z}$ and $I\subset\mathbb{R}$ be (possibly infinite) intervals. Let $\mathcal{L}$ be a random continuous function defined on $\Lambda\times I$ which is in the class $\mathbf{LC}$. For any $j\in\Lambda$ and $t\in I$, the mean of $\mathcal{L}_j(t)$, denoted by $\overline{\mathcal{L}}_j(t)$, exists. Let $\hat{\mathcal{L}}_j(t)= \mathcal{L}_j(t)-\overline{\mathcal{L}}_j(t)$. Then the following statements hold.
\begin{enumerate}[label=(\roman*)]
  \item For all $p\geq 1$, $  j\in \Lambda$ and $t,s\in I$, we have
\begin{align*}
\mathbb{E}|\hat{\mathcal{L}}_j(t)-\hat{\mathcal{L}}_j(s)|^p\leq N_p|t-s|^{p/2},
\end{align*}
where $N_p=\mathbb{E}|\mathcal{N}|^p$ for a standard normal $\mathcal{N}$.
  \item There exist universal constants $C_1,C_2>0$ such that for all $[a,b]\subset I$, $j\in \Lambda$ and $K\geq 0$, we have
\begin{align*} 
\mathbb{P}\left( \sup_{t,s\in [a,b], t\neq s}  \frac{|\hat{\mathcal{L}}_j(t)-\hat{\mathcal{L}}_j(s) |}{\sqrt{|t-s| \log (2(b-a)/|t-s|)}} > K \right)\leq C_1e^{-C_2K^2}.
\end{align*}
  \item For any $p>0$ and $\alpha\in (0,\frac{1}{2})$, there exists a constant $C(\alpha,p)>0$ such that for all $[a,b]\subset I$ and $j\in \Lambda$, we have  
  \begin{align*}
  \mathbb{E} \|\hat{\mathcal{L}}_j\|_{\alpha,[a,b]}^p  \leq  C(\alpha,p)(b-a)^{p/2-\alpha p }. 
\end{align*}   
\end{enumerate}
\end{theorem}

Apart from $\beta$-Dyson Brownian motions, there are several other important $\beta$ ensembles originating from random matrix theory, such as the $\beta$-Laguerre processes, the $\beta$-Jacobi processes, and the $\beta$-Dyson Brownian motions with general potentials. We refer readers to \cite{huang2024convergenceframeworkairybetaline} and references therein. While we believe our methods could be applied to some of those, we do not explore these directions in the current paper.

In the next proposition, we provide a general comparison of the expectations of a convex functional $G$ with respect to $\gamma_{\mathsf{s}}$ and $\mu_{\mathsf{s},\mathbf{H}}$. A functional $G$ on $C(\llbracket k,\ell \rrbracket\times [a,b],\mathbb{R})$ is said to be convex and dependent on finitely many times if there exist a finite set $\Pi\subset [a,b]$ and a convex function $g$ on $\mathbb{R}^{\llbracket k,\ell \rrbracket\times\Pi}$ such that $G=g\circ \textup{Res}_{\llbracket k,\ell \rrbracket\times\Pi}$. Here $\textup{Res}_{\llbracket k,\ell \rrbracket\times\Pi}$ is the restriction map from $C(\llbracket k,\ell \rrbracket\times [a,b],\mathbb{R})$ to $\mathbb{R}^{\llbracket k,\ell \rrbracket\times \Pi}$ given by $z\mapsto\{z_j(t)\}_{(j,t)\in \llbracket k,\ell \rrbracket\times \Pi}$.
\begin{proposition}\label{pro:main}
Fix $R=\llbracket k,\ell \rrbracket\times [a,b]\subset \mathbb{Z}\times \mathbb{R}$, $\mathsf{s}=(R,x,\eta(t))\in\mathsf{S}^{R}$ and $\mathbf{H}\in\mathsf{H}^{R}$. Let $\gamma= \gamma_{\mathsf{s}}$ be the Wiener measure on $C(\llbracket k,\ell \rrbracket\times [a,b],\mathbb{R})$ associated with $\mathsf{s}$ and let $\mu=\mu_{\mathsf{s},\mathbf{H}}$ be the probability measure that is proportional to $e^{-\mathbf{H}(z)}\gamma_{\mathsf{s}}$. We write $\bar{z}_{\gamma}(t)$ and $\bar{z}_{\mu}(t)$ for the mean values of $z(t)$ under $\gamma$ and $\mu$ respectively. Then for any convex functional $G$ on $C(\llbracket k,\ell \rrbracket\times [a,b],\mathbb{R})$ that depends on finitely many times, we have  
\begin{equation}\label{equ:main}
\mathbb{E}_{\mu}[G(z-\bar{z}_\mu)]\leq\mathbb{E}_{\gamma}[G(z-\bar{z}_\gamma)].  
\end{equation}
Fix $\breve{\mathsf{s}}\in\breve{\mathsf{S}}^{R}$. Let $\breve{\gamma}= \gamma_{\breve{\mathsf{s}}}$ be the Wiener measure on $C(\llbracket k,\ell \rrbracket\times [a,b],\mathbb{R})$ associated with $\breve{s}$ and let $\breve{\mu}=\mu_{\breve{\mathsf{s}},\mathbf{H}}$ be the probability measure that is proportional to $e^{-\mathbf{H}(z)}\gamma_{\breve{\mathsf{s}}}$. We write $\bar{z}_{\breve{\gamma}}(t)$ and $\bar{z}_{\breve{\mu}}(t)$ for the mean values of $z(t)$ under $\breve{\gamma}$ and $\breve{\mu}$ respectively. Then for any convex functional $G$ on $C(\llbracket k,\ell \rrbracket\times [a,b],\mathbb{R})$ that depends on finitely many times, we have  
\begin{equation}\label{equ:mainb}
\mathbb{E}_{\breve{\mu}}[G(z-\bar{z}_{\breve{\mu}})]\leq\mathbb{E}_{\breve{\gamma}}[G(z-\bar{z}_{\breve{\gamma}})].  
\end{equation}

{\color{black}
\begin{proof}
We present the proof for \eqref{equ:main}. The argument for the bridge counterpart \eqref{equ:mainb} is similar. Let $\Pi\subset [a,b]$ be a finite set such that $G=g\circ\text{Res}_{\llbracket k,\ell \rrbracket\times\Pi}$ for some convex function $g$ on $\mathbb{R}^{\llbracket k,\ell \rrbracket\times\Pi}$. Without loss of generality, we may assume $\Pi\subset (a,b]$. Let $\hat{\gamma}$ and $\hat{\mu}$ be the probability measures on $\mathbb{R}^{\llbracket k,\ell \rrbracket\times\Pi }$ which are marginal laws of $\gamma$ and $\mu$ on $\{ z_j(t) \}_{ (j,t)\in\llbracket k,\ell \rrbracket \times\Pi}$ respectively. Lemma~\ref{lem:main} implies $\frac{d\hat{\mu}}{d\hat{\gamma}}$ is log-concave. We write $w$ for a vector in $\mathbb{R}^{\llbracket k,\ell \rrbracket \times\Pi}$ and let $\bar{w}_\mu$ and $\bar{w}_\gamma$ be the mean values of $w$ under $\hat{\mu}$ and $\hat{\gamma}$ respectively. Applying \cite[Theorem 1.1]{MR2095937}, we get 
\begin{align*}
\mathbb{E}_{\hat{\mu}}[g(w-\bar{w}_\mu)]\leq\mathbb{E}_{\hat{\gamma}}[g(w-\bar{w}_\gamma)]. 
\end{align*}
Then \eqref{equ:main} holds true because $\mathbb{E}_{\hat{\mu}}[g(w-\bar{w}_\mu)]=\mathbb{E}_{\mu}[G(z-\bar{z}_\mu)]$ and $\mathbb{E}_{\hat{\gamma}}[g(w-\bar{w}_\gamma)]=\mathbb{E}_{\gamma}[G(z-\bar{z}_\gamma)]$. 
\end{proof}
}
\end{proposition}

With Proposition~\ref{pro:main}, we prove Theorem~\ref{thm:2}, the main result of the paper.

\begin{proof}[Proof of Theorem~\ref{thm:2} (i)]
{\color{black}From Corollary~\ref{cor:arginal_logconcave}, the marginal law of $\mathcal{L}_j(t)$ is log-concave. This ensures the mean of $\mathcal{L}_j(t)$ exists.} Applying Proposition~\ref{pro:main} with $G(z)=|z_j(t)-z_j(s)|^p$, we deduce that Theorem~\ref{thm:2} (i) holds true provided the law of $\mathcal{L}$ is of the form $\mu_{\mathsf{s},\mathbf{H}}$ or $\mu_{\breve{\mathsf{s}},\mathbf{H}}$. By taking a limit, we conclude that Theorem~\ref{thm:2} (i) holds true for all $\mathcal{L}\in \mathbf{LC}$.
\end{proof}
\begin{proof}[Proof of Theorem~\ref{thm:2} (ii)]
Theorem~\ref{thm:2} (i) shows that the $p$-th moment of $|\hat{\mathcal{L}}_j(t)-\hat{\mathcal{L}}_j(s)|/|t-s|^{1/2}$ is less then the one of a standard normal random variable. Therefore, for all $a\in [0,1/2)$, we have 
\begin{align*}
\mathbb{E}\left[ e^{a |\hat{\mathcal{L}}_j(t)-\hat{\mathcal{L}}_j(s)|^2/|t-s|} \right] \leq (1-2a)^{-1/2}.  
\end{align*}
For any $K\geq 1$, we take $a=\frac{1}{2}-\frac{1}{2K^2}$ and apply the Markov inequality to get 
\begin{align}\label{equ:tail}
\mathbb{P}\left( |\hat{\mathcal{L}}_j(t)-\hat{\mathcal{L}}_j(s)|>K(t-s)^{1/2} \right)
\leq e^{-K^2/2+\log K+1/2}.
\end{align}
The assertion then follows by combining \eqref{equ:tail} and \cite[Lemma 3.3]{MR4260467}.
\end{proof}
\begin{proof}[Proof of Theorem~\ref{thm:2} (iii)]
Fix $\alpha\in (0,\frac{1}{2})$. By Jensen's inequality, it suffices to prove the assertion for $p>\left( \frac{1}{2}-\alpha \right)^{-1} $, which is assumed henceforth. From the Garsia-Rodemich-Rumsey inequality \cite[Lemma 1.1]{MR267632}
, there exists a constant $C'(\alpha,p)$ depending only on $\alpha$ and $p$ such that for all $a<b$ and $j\in\mathbb{N}$,
\begin{align*}
\|\hat{\mathcal{L}}_j\|^p_{ \alpha,[a,b]}\leq C'(\alpha,p) \left( \int_a^b\int_a^b \frac{|\hat{\mathcal{L}}_j(t)-\hat{\mathcal{L}}_j(s)|^p}{|t-s|^{\alpha p+2}}\, dtds \right).
\end{align*}
Taking the expectation on both sides and applying Theorem~\ref{thm:2} (i), we conclude
\begin{align*}
\mathbb{E}\|\hat{\mathcal{L}}_j\|^p_{ \alpha,[a,b]}\leq & C'(\alpha,p)  N_p \left( \int_a^b\int_a^b |t-s|^{p/2-\alpha p-2}\, dtds \right) =  C(\alpha,p) (b-a)^{p/2-\alpha p}
\end{align*}
with a suitably chosen constant $C(\alpha,p)$.
\end{proof}

To prove Lemma~\ref{lem:main}, we need the next two lemmas.
\begin{lemma}\label{lem:conditionlogconcave}
Fix $n,m\in\mathbb{N}$. Let $\hat{\gamma}$ be a Gaussian measure on $\mathbb{R}^{n+m}$. We write $x\in\mathbb{R}^n$, $y\in\mathbb{R}^m$ and $\hat{\gamma}=\rho(x,y)\, dxdy $.  Then $
 \rho(y\, |\, x)= \rho(x,y)\big/ \int_{\mathbb{R}^m} \rho(x,\tilde{y} )\, d\tilde{y}   $ is log-concave in $(x,y)$.
 \begin{proof} 
 We have
 $\rho(x,y)= c e^{-Q(x,y)-\left\langle a,x \right\rangle-\left\langle b,y \right\rangle} $ for some positive quadratic form $Q$, $a\in\mathbb{R}^n$, $b\in\mathbb{R}^m$ and $c>0$. Through completing the square, we have
 $Q(x,y)+\left\langle a,x \right\rangle+\left\langle b,y \right\rangle=Q_1(x+a')+Q_2(y+y'(x)+b')+c'$ for some positive quadratic forms $Q_1,Q_2$, $a'\in\mathbb{R}^n$, $b'\in\mathbb{R}^m, c'\in\mathbb{R}$ and some linear map $y':\mathbb{R}^n\to\mathbb{R}^m$. Then $\rho(y\,|\,x)=c''e^{-Q_2(y+y'(x)+b')}$, which is log-concave.
\end{proof}
\end{lemma}
\begin{lemma}\label{lem:logconvexmarginal_b}
Fix $R=\llbracket k,\ell \rrbracket\times [a,b]\subset \mathbb{Z}\times \mathbb{R}$, $\breve{\eta}(t)\in C_{0,0}(\llbracket k,\ell \rrbracket\times [a,b],\mathbb{R})$, and $\mathbf{H}\in\mathsf{H}^{R}$. For any $x,y\in\mathbb{R}^{\llbracket k,\ell \rrbracket}$, we let $\breve{\mathsf{s}}(x,y)=(R,x,y,\breve{\eta}(t))\in \breve{\mathsf{S}}^{R}$ and let $\mathbb{E}_{x,y}$ be the expectation with respect to $\gamma_{\breve{\mathsf{s}}(x,y)}$. Define  
\begin{align*}
Z_{\mathbf{H}}(x,y)=\mathbb{E}_{x,y}\big[e^{-\mathbf{H}(z)}\big].
\end{align*}
Then $Z_{\mathbf{H}}(x,y)$ is a log-concave function in $(x,y)\in \mathbb{R}^{\llbracket k,\ell \rrbracket}\times\mathbb{R}^{\llbracket k,\ell \rrbracket}$.
\begin{proof}
Let $\mathbf{H}$ be given by the form \eqref{equ:H}. From Lemma~\ref{lem:mu0}, we may assume $f_i$ and $F$ are non-negative. For $n\in\mathbb{N}$, set $\tau_{n,i}=a+\frac{(b-a)i}{n}$, 
\begin{align*}
\mathbf{H}_n(z)=\sum_{i=1}^m f_i(z(t_i))+ \frac{b-a}{n} \sum_{i=1}^n  F\left(\tau_{n,i} ,z(\tau_{n,i})\right),
\end{align*}
and
\begin{align*}
Z_{\mathbf{H}_n}(x,y)=\mathbb{E}_{x,y}\big[e^{-\mathbf{H}_n(z)}\big].
\end{align*}
For any $z\in C(\llbracket k,\ell \rrbracket \times [a,b],\mathbb{R})$, we have $\lim_{n\to\infty}\mathbf{H}_n(z)\to \mathbf{H}(z)$. The non-negativity of $f_i$ and $F$ ensures $e^{-\mathbf{H}_n(z)}\in (0,1)$. Applying the bounded convergence theorem, we have $\lim_{n\to\infty} Z_{\mathbf{H}_n}(x,y)=Z_{\mathbf{H}}(x,y)$. Hence it suffices to show $Z_n(x,y)$ is log-concave for all $n\in\mathbb{N}$.

Fix $n\in\mathbb{N}$ and set $\Pi= \{t_i\}_{i=1}^m \cup \{ a+\frac{(b-a)i}{n} \}_{i=1}^n  \setminus \{a,b\} $. The marginal law of $\gamma_{\breve{s}(x,y)}$ on $\{ z_j(t) \}_{ (j,t)\in\llbracket k,\ell \rrbracket \times\Pi}$ is a Gaussian measure on $\mathbb{R}^{\llbracket k,\ell \rrbracket \times\Pi}$ and we denote such a measure by $\hat{\gamma}(x,y)$. We use $w$ to denote a vector in $\mathbb{R}^{\llbracket k,\ell \rrbracket \times\Pi }$. Let $\rho(w|x,y)$ be the density of $\hat{\gamma}(x,y)$ with respect to the Lebesgue measure $dw$. From Lemma~\ref{lem:conditionlogconcave}, $\rho(w|x,y)$ is log-concave in $(x,y,w)$. Then $Z_{\mathbf{H}_n}(x,y)$ is the integral of a log-concave function with respect to $\hat{\gamma}(x,y)=\rho(w|x,y)dw$. Hence $Z_{\mathbf{H}_n}(x,y)$ is log-concave because of the Pr{\'e}kopa–Leindler Theorem \cite[Theorem 	13.9]{MR2814377}. This finishes the proof.

\end{proof} 
\end{lemma}
The argument for the following lemma is similar to the one for Lemma~\ref{lem:logconvexmarginal_b} and we omit the proof.

\begin{lemma}\label{lem:logconvexmarginal}
Fix $R=\llbracket k,\ell \rrbracket\times [a,b]\subset \mathbb{Z}\times \mathbb{R}$, $\eta(t)\in C_0(\llbracket k,\ell \rrbracket\times [a,b],\mathbb{R})$, and $\mathbf{H}\in\mathsf{H}^{R}$. For any $x\in\mathbb{R}^{\llbracket k,\ell \rrbracket}$, we let $  \mathsf{s} (x )=(R,x,\eta(t))\in  \mathsf{S}^{R}$ and let $\mathbb{E}_{x}$ be the expectation with respect to $\gamma_{ \mathsf{s}(x)}$. Define  
\begin{align*}
Z_{\mathbf{H}}(x )=\mathbb{E}_{x }\big[e^{-\mathbf{H}(z)}\big].
\end{align*}
Then $Z_{\mathbf{H}}(x )$ is a log-concave function in $x\in \mathbb{R}^{\llbracket k,\ell \rrbracket} $.
\end{lemma}

\begin{proof}[Proof of Lemma~\ref{lem:main}] 
{\color{black}
We start by recalling the setting. Fix $R=\llbracket k,\ell \rrbracket\times [a,b]\subset \mathbb{Z}\times \mathbb{R}$, $\mathsf{s}=(R,x,\eta(t))\in\mathsf{S}^{R}$, $\mathbf{H}\in\mathsf{H}^{R}$ and a finite set $\Pi\subset (a,b]$. We write 
 $\hat{\gamma}$ and $\hat{\mu}$ for the marginal laws of $\gamma_{\mathsf{s}}$ and $\mu_{\mathsf{s},\mathbf{H}}$ on $\{ z_j(t) \}_{ (j,t)\in\llbracket k,\ell \rrbracket \times\Pi}$  respectively. The goal is to prove the log-concavity of  $\frac{d\hat{\mu}}{d\hat{\gamma}}$. 
 
Without loss of generality, we may assume $b\in \Pi$.
} We denote $\Pi=\{\tau_1<\dots<\tau_q=b\}$ with $a<\tau_1$ and adopt the convention that $\tau_0=a$. We identify $z(t)\in C(\llbracket k,\ell \rrbracket\times [a,b],\mathbb{R})$ with $$((w_{(0)},w_{(1)},\dots,w_{(q)}),(z_{(1)},z_{(2)},\dots,z_{(q)}))\in\prod_{i=0}^q\mathbb{R}^{\llbracket k,\ell \rrbracket}\times \prod_{i=1}^q C(\llbracket k,\ell \rrbracket\times[\tau_{i-1},\tau_i],\mathbb{R})$$ through
\begin{align}\label{equ:identification}
w_{(i)}=z(\tau_i)\in\mathbb{R}^{\llbracket k,\ell \rrbracket},\ z_{(i)}=z|_{[\tau_{i-1},\tau_i]}\in C(\llbracket k,\ell \rrbracket\times [\tau_{i-1},\tau_i],\mathbb{R}).
\end{align}
Recall that $\mathsf{s}=(R,x,\eta(t))$. In particular, $z(a)=x$ with probability one under $\gamma=\gamma_{\mathsf{s}}$ and $\mu=\mu_{\mathsf{s},\mathbf{H}}$. From now on we adopt the convention $w_{(0)}\equiv x$. Define $\eta_{(i)}(t)\in C_{0,0}(\llbracket k,\ell \rrbracket\times [\tau_{i-1},\tau_i],\mathbb{R})$ by $$\eta_{(i)}(t)=\eta(t)-\frac{t-\tau_{i-1}}{\tau_i-\tau_{i-1}}\eta(\tau_i)-\frac{\tau_{i}-t}{\tau_i-\tau_{i-1}}\eta(\tau_{i-1}), $$
and set $\breve{\mathsf{s}}(w_{(i-1)},w_{(i)})=(\llbracket k,\ell \rrbracket\times [\tau_{i-1},\tau_i],w_{(i-1)},w_{(i)},\eta_{(i)})$.  Under the identification \eqref{equ:identification}, $ \gamma $ can be rewritten as
\begin{align*}
\gamma= \hat{\gamma} \times\prod_{i=1}^q   \gamma_{  \breve{\mathsf{s}}(w_{(i-1)},w_{(i)})} 
\end{align*}
Moreover, for suitably chosen convex Hamiltonians $\mathbf{H}_{(i)}:C(\llbracket k,\ell \rrbracket\times [\tau_{i-1},\tau_i],\mathbb{R})\to \mathbb{R}$, we have
$$ e^{-\mathbf{H}(z)}=\prod_{i=1}^q e^{-\mathbf{H}_{(i)}(z_{(i)})} $$
under the identification~\eqref{equ:identification}. Therefore, for $w=(w_{(1)},\dots,w_{(q)})\in\mathbb{R}^{\Lambda\times\Pi}$, 
\begin{align*}
\frac{d\hat{\mu}}{d\hat{\gamma}}(w)\propto \prod_{i=1}^q 
\mathbb{E}_{w_{(i-1)},w_{(i)}}[e^{-\mathbf{H}_{(i)}(z_{(i)})}],
\end{align*}
where $\mathbb{E}_{w_{(i-1)},w_{(i)}}$ is the expectation with respect to $\breve{\mathsf{s}}(w_{(i-1)},w_{(i)})$. Apply Lemma~\ref{lem:logconvexmarginal_b}, we conclude that $\frac{d\hat{\mu}}{d\hat{\gamma}}$ is log-concave. This finishes the proof.

\end{proof}

\section{$\beta$-Dyson Brownian Motion and other examples}\label{sec:3}

In Subsection~\ref{sec:3.1}, we show that $\beta$-Dyson Brownian motions belong to $\mathbf{LC}$ for $\beta\geq 2$. This allows us to prove modulus of continuity estimates for $\beta$-Dyson Brownian motions (Theorem~\ref{thm:1}) and for Air$\textup{y}_\beta$ line ensembles (Corollary~\ref{prop:3}) for $\beta \geq 2$. In Subsection~\ref{sec:3.2}, we provide two more examples in $\mathbf{LC}$, the O'Connell-Yor line ensemble and the KPZ line ensemble, and prove Corollary~\ref{prop:5}, a modulus of continuity estimate for the KPZ line ensemble. 

\subsection{$\beta$-Dyson Brownian Motion}\label{sec:3.1}
In this subsection, we prove the following proposition.

\begin{proposition}\label{prop:DBM}
For $\beta\geq 2$, the $\beta$-Dyson Brownian motion starting at $x_*\in\overline{\mathbb{W}}^N$ is in the class $\mathbf{LC}$.
\end{proposition}
With Proposition~\ref{prop:DBM}, we present the proofs for Theorem~\ref{thm:1} and Corollary~\ref{prop:3}.
\begin{proof}[Proof of Theorem~\ref{thm:1}]
This is a direct consequence of Theorem~\ref{thm:2} and Proposition~\ref{prop:DBM}.
\end{proof}
\begin{proof}[Proof of Corollary~\ref{prop:3}]
Since the Air$\textup{y}_\beta$ line ensemble is the distributional limit of $\beta$-Dyson Brownian motions under the scaling \eqref{equ:Airybeta} \cite{landon2020edgescalinglimitdyson,huang2024convergenceframeworkairybetaline,gorin2024}, we have from Proposition~\ref{prop:DBM} that $\sqrt{\beta/2}\mathcal{A}^{\beta}(t)$ is in the class $\mathbf{LC}$. Then the assertion follows from Theorem~\ref{thm:2} and the stationarity of $\mathcal{A}^{\beta}(t)$.
\end{proof}
In order to prove Proposition~\ref{prop:DBM}, we first calculate the Radon-Nikodym derivative of the law of a $\beta$-Dyson Brownian motion with respect to a Wiener measure. Recall that $h_\beta$ and $V_\beta$ are defined in \eqref{def:h_beta} and \eqref{def:V_beta} as
\begin{align*} 
h_\beta(x)=\left\{\begin{array}{cc}
\prod_{1\leq i<j\leq N}|x_i-x_j|^{\beta/2}, & x\in\mathbb{W}^N,\\
0, & x\notin \mathbb{W}^N,
\end{array} \right.
\end{align*}
and
\begin{align*} 
V_\beta(x)=\left\{ \begin{array}{cc}
\frac{\beta(\beta-2)}{4}\sum_{1\leq i<j\leq N}\frac{1}{(x_i-x_j)^2}, & x\in\mathbb{W}^N,\\
\infty, & x\notin \mathbb{W}^N.
\end{array} \right.
\end{align*}

\begin{lemma}\label{lem:DBMGirsanov}
Let $X(t)$ be the $\beta$-Dyson Brownian motion starting at  $x_*\in\mathbb{W}^N$ with $\beta\geq 1$. Fix $t_0>0$, let $\mu$ be the marginal law of $X(t)$ on $t\in [0,t_0]$ and let $\gamma$ be the marginal law on $[0,t_0]$ of independent Brownian motions starting at $x_*$. Then $\mu$ is absolutely continuous with respect to $\gamma$. Moreover, the Radon-Nikodym derivative is given by
\begin{align}\label{equ:RN_DBM}
\frac{d\mu}{d\gamma}(z)=\left\{ \begin{array}{cc}
\displaystyle \frac{h_\beta(z(t_0))}{h_{\beta}(z(0))}\exp\left( -\int_0^{t_0} V_\beta(z(s))\, ds \right), & z(t)\in \mathbb{W}^N\ \textup{for all}\ t\in [0,t_0],\\
0, & \textup{otherwise}.
\end{array} \right.
\end{align}
\begin{proof}
Let $X(t)=(X_1(t),X_2(t),\dots, X_N(t))$ be an $N$-dimensional Brownian motion starting at $x_*$ defined on a filtered probability space $(\Omega,\mathcal{F}_t,\mathbb{P}_0)$. For $n\geq 1$, define the stopping times
$$\zeta_n= \inf\left\{t\geq 0 \, :\, \min_{1\leq i\leq N-1}\left| X_i(t)-X_{i+1}(t)\right|\leq 1/n\ \textup{or}\ \max_{1\leq i\leq N}|X_i(t)|\geq n\right\}.$$
From the above definition, $\zeta_n$ is non-decreasing in $n$ and we write $\zeta=\lim_{n\to\infty}\zeta_n$. It is direct to check that $$\zeta=\inf\left\{t\geq 0 \, :\, \min_{1\leq i\leq N-1}\left| X_i(t)-X_{i+1}(t)\right|= 0\right\}.$$
Define 
\begin{align}\label{equ:Mn}
M^{(n)}(t)=\frac{h_\beta(X(t\wedge \zeta_n ))}{h_{\beta}(X(0))}\exp\left( -\int_0^{t\wedge \zeta_n} V_\beta(X(s))\, ds \right).
\end{align}
A direct calculation yields $\Delta h_\beta(x)=2V_\beta(x) h_\beta(x)$ for $x\in\mathbb{W}^N$. Applying It{\^o}'s formula, we deduce $dM^{(n)}(t)=M^{(n)}(t)dL^{(n)}(t)$ with 
\begin{align*}
L^{(n)}(t)=\int_0^{t\wedge \zeta_n}
\sum_{i=1}^N\left(   \sum_{j\neq i}\frac{\beta/2 }{X_i(s)-X_j(s)}  \right) dX_i(s).
\end{align*}
This implies $M^{(n)}(t)$ is a $\mathbb{P}_0$-local martingale. From the definition of $\zeta_n$, $M^{(n)}(t)$ is locally bounded in $t$. Hence $M^{(n)}(t)$ is a true $\mathbb{P}_0$-martingale. Let $\mathbb{P}^{(n)}$ be the measure obtained through tilting $\mathbb{P}_0$ by $M^{(n)}(t)$. From the Girsanov theorem \cite[Theorem 1.7, Chapter VIII] {MR1725357}, there is a $ {\mathbb{P}}^{(n)}$-Brownian motion $B^{(n)}(t)$ such that $X(t)$ solves \eqref{equ:DBM} for $t\leq \zeta_n$ with $B(t)$ replaced by $B^{(n)}(t)$. Since a $\beta$-Dyson Brownian motion does not collide for $\beta\geq 1$ \cite[Lemma 1]{MR1217451} , we have $\lim_{n\to\infty} {\mathbb{P}}^{(n)}(\zeta_n\leq t)=0 $ for any $t\in [0,\infty)$. From \cite[Theorem 1.3.5]{MR2190038}, there exists a probability measure $ {\mathbb{P}}$ such that $ {\mathbb{P}}= {\mathbb{P}}^{(n)}$ on $\mathcal{F}_{\zeta_n}$. This implies there exists a $ {\mathbb{P}}$-Brownian motion $B(t)$ such that $X(t)$ solves \eqref{equ:DBM}. Moreover, from \eqref{equ:Mn} we have
\begin{align*}
\frac{d {\mathbb{P}}}{d\mathbb{P}_0}\bigg|_{\mathcal{F}_{t_0}} =\left\{ \begin{array}{cc}
\displaystyle\frac{h_\beta(X(t_0))}{h_{\beta}(X(0))}\exp\left( -\int_0^{t_0} V_\beta(X(s))\, ds \right), & t_0<\zeta,\\
0, & \zeta\leq t_0.
\end{array} \right.
\end{align*}
for any $t_0> 0$. This concludes \eqref{equ:RN_DBM}.
\end{proof}
\end{lemma}

\begin{proof}[Proof of Proposition~\ref{prop:DBM}]
Since $\mathbf{LC}$ is closed under the convergence in distribution, we may assume $x_*\in\mathbb{W}^N$ without loss of generality \cite[Proposition 4.3.5]{MR2760897}. From now on, we assume $X(t)$ is the $\beta$-Dyson Brownian motion starting at $x_*\in\mathbb{W}^N$ with $\beta\geq 2$. Our goal is to prove $X(t)$ belongs to $\mathbf{LC}$. 

It suffices to show that for all $t_0>0$, the marginal law of $X(t)$ on $t\in[0,t_0]$, denoted by $\mu$, belongs to $\mathfrak{M}$. From Lemma~\ref{lem:DBMGirsanov}, the Radon-Nikodym derivative of $\mu$ with respect to a Wiener measure $\gamma$ is given by
\begin{align*}
\frac{d\mu}{d\gamma}(z)=\left\{ \begin{array}{cc}
\displaystyle \frac{h_\beta(z(t_0))}{h_{\beta}(z(0))}\exp\left( -\int_0^{t_0} V_\beta(z(s))\, ds \right), & z(t)\in \mathbb{W}^N\ \textup{for all}\ t\in [0,t_0],\\
0, & \textup{otherwise}.
\end{array} \right.
\end{align*}
For $\beta\geq 2$, $-\log h_\beta(x)$ and $V_\beta(x)$ are both convex for $x\in\mathbb{W}^N$. We now approximate $-\log h_\beta(x)$ and $V_\beta(x)$ by convex functions on $\mathbb{R}^N$.  For $\varepsilon\in (0,1)$, set
\begin{align*}
q_{1,\varepsilon}(s)=\left\{ \begin{array}{cc}
-\log s, & s\geq \varepsilon,\\
-\log \varepsilon-\varepsilon^{-1}(s-\varepsilon), & s<\varepsilon,
\end{array} \right. 
\end{align*} 
and
\begin{align*}
q_{2,\varepsilon}(s)=\left\{ \begin{array}{cc}
s^{-2}, & s\geq \varepsilon,\\
  \varepsilon^{-2}-2\varepsilon^{-3}(s-\varepsilon), & s<\varepsilon,
\end{array} \right.\qquad  q_{3,\varepsilon}(s)=\left\{ \begin{array}{cc}
0, & s\geq 0,\\
  -\varepsilon^{-1}s, & s<0.
\end{array} \right.
\end{align*}
Define $h_{\beta,\varepsilon}(x)$ such that for $x\in\mathbb{R}^N$, $-\log h_{\beta,\varepsilon}(x)= \frac{\beta}{2}\sum_{i<j}q_{1,\varepsilon}(x_i-x_j).$ Also define $V_{\beta,\varepsilon}(x)$ by
\begin{align}\label{equ:04010127}
\ V_{\beta,\varepsilon}(x)=\left\{\begin{array}{cc}
4^{-1}{\beta(\beta-2)} \sum_{i<j}q_{2,\varepsilon}(x_i-x_j), & \beta>2,\\
\sum_{i<j}q_{3,\varepsilon}(x_i-x_j), & \beta=2.
\end{array} \right.  
\end{align}
Let $\mathbf{H}_\varepsilon$ be the corresponding Hamiltonian
\begin{align*}
\mathbf{H}_\varepsilon(z)=-\log h_{\beta,\varepsilon}(z(t_0))+\int_0^{t_0} V_{\beta,\varepsilon}(z(s))\, ds 
\end{align*}
and let $\mu_\varepsilon$ be the probability measure defined through $d\mu_\varepsilon/d\gamma\propto e^{-\mathbf{H}_\varepsilon}$. Since $-\log h_{\beta,\varepsilon}(x)$ and $V_{\beta,\varepsilon}(x)$ are convex, $\mu_\varepsilon\in\mathring{\mathfrak{M}}$. It is straightforward to check that for $\gamma$-a.s. $z\in C(\llbracket 1,N \rrbracket\times[0,t_0],\mathbb{R})$, $e^{-\mathbf{H}_{\varepsilon}(z)}\searrow h_\beta(x_*)\frac{d\mu}{d\gamma}(z)$ when $\varepsilon\to 0+$. It then follows from the dominated convergence theorem that $\mu_\varepsilon$ converges weakly to $\mu$. Hence $\mu\in\mathfrak{M}$ and $X(t)$ belongs to $\mathbf{LC}$.
\end{proof} 


{\color{black}\begin{lemma}
Dyson Brownian bridges, i.e., Brownian bridges conditioned to be non-intersecting, are in the class $\mathbf{LC}$.
\begin{proof}
By approximation, it suffices to prove that Dyson Brownian bridges with strictly ordered boundary data are in the class $\mathbf{LC}$. Fix $R = \llbracket k, \ell \rrbracket \times [a, b] \subset \mathbb{Z} \times \mathbb{R}$, $x = (x_k, x_{k+1}, \dots, x_\ell)$, and $y = (y_k, y_{k+1}, \dots, y_\ell)$ with 
$x_k > x_{k+1} > \dots > x_\ell$ and $ \quad y_k > y_{k+1} > \dots > y_\ell. $ Set $\breve{s} = (R, x, y, 0)$. The law of independent Brownian bridges on $[a, b]$ with boundary data $x$ and $y$ is given by $\gamma_{\breve{s}}$. For $\varepsilon \in (0, 1)$, take $V_{2, \varepsilon}(s)$ as in \eqref{equ:04010127} and $\mathbf{H}_\varepsilon(z) = \int_a^b V_{2, \varepsilon}(z(s)) \, ds$. The convexity of $V_{2, \varepsilon}$ implies $\mu_{\breve{s}, \mathbf{H}_\varepsilon}$ belongs to $\mathring{\mathfrak{M}}$. By taking $\varepsilon \to 0$, we deduce that the Dyson Brownian bridges with strictly ordered boundary data belong to $\mathbf{LC}$.  
\end{proof}
\end{lemma}}


\subsection{O'Connell-Yor line ensemble and KPZ line ensemble}\label{sec:3.2}
In this subsection, we show that both the O'Connell-Yor line ensemble and the KPZ line ensemble belong to $\mathbf{LC}$. As a consequence, we prove Corollary~\ref{prop:5}, a modulus of continuity estimate for the KPZ line ensemble.

Fix $N\geq 2$. The O'Connell-Yor line ensemble \cite{MR2952082} with index $\lambda=(\lambda_1,\dots,\lambda_N)\in\mathbb{R}^N$, denoted by $Y^{\lambda}(t)$, is the diffusion process on $\mathbb{R}^N$ obtained by applying the geometric RSK correspondence to the $N$-dimensional Brownian motion with drift $\lambda$. It is proved in \cite[Theorem 3.1]{MR2952082} that $Y^{\lambda}(t)$ solves the SDE
\begin{align}\label{equ:OClambda}
dY_j^{\lambda}(t)=\nabla_j\log \psi_\lambda (Y^{\lambda}(t))dt+dB_j(t), 1\leq j\leq N,
\end{align}
where $\psi_\lambda$ is $\mathfrak{gl}_{N}$-Whittaker function. The Whittaker function $\psi_\lambda$ is an eigenfunction of the quantum Toda Hamiltonian
\begin{align*}
\frac{1}{2}\Delta \psi_\lambda (x) = \left( \sum_{i=1}^{N-1} e^{x_{i+1}-x_i}+\frac{1}{2}\sum_{i=1}^N\lambda_i^2\right) \psi_\lambda	(x).
\end{align*}
We refer the readers to \cite[Section 2]{MR2952082} for a more detailed description of $\psi_\lambda$.
\begin{proposition}\label{prop:OY}
Fix $N\in\mathbb{N}$ and $\lambda\in\overline{\mathbb{W}}^N$. Then the O'Connell-Yor line ensemble with index $\lambda$ is in the class $\mathbf{LC}$.
\end{proposition}
With Proposition~\ref{prop:OY}, we can prove Corollary~\ref{prop:5}.
\begin{proof}[Proof of Corollary~\ref{prop:5}]
It is known \cite{CH16,Nic21} that the KPZ line ensemble is the distributional limit of the O'Connell-Yor line ensemble with index $\lambda=\vec{0}$. Together with Proposition \ref{prop:OY}, the KPZ line ensemble belongs to $\mathbf{LC}$. Then the assertion follows from Theorem~\ref{thm:2} and the stationarity of the KPZ line ensemble along the parabola $-2^{-1}t^2$. 
\end{proof}

To prove Proposition~\ref{prop:OY}, we begin by showing the solution to \eqref{equ:OClambda} belongs to $\mathbf{LC}$ provided $\lambda\in\overline{\mathbb{W}}^N$ and it starts at some $x_*\in\mathbb{R}^N$, 

\begin{lemma}\label{lem:qtinm}
Fix $N\in\mathbb{N}$, $x_*\in\mathbb{R}^N$, and $\lambda\in\overline{\mathbb{W}}^N$. Let $X^\lambda(t)$ be the solution to \eqref{equ:OClambda} that starts at $x_*$. Then $X^\lambda(t)$ is in the class $\mathbf{LC}$.
\end{lemma}

\begin{proof}
Since $\mathbf{LC}$ is closed under the convergence in distribution, we may assume $\lambda\in\mathbb{W}^N$ without loss of generality. Let $X(t)=(X_1(t),X_2(t),\dots, X_N(t))$ be an $N$-dimensional Brownian motion starting at $x_*$ defined on a filtered probability space $(\Omega,\mathcal{F}_t,\mathbb{P}_0)$. We write $X^\lambda(t)=X(t)+\lambda t$ for the Brownian motion with drift $\lambda$.

Set $F(x)=\sum_{i=1}^{N-1} e^{x_{i+1} -x_i }. $  It is proved in \cite[Proposition 2.2]{MR2884226} that
\begin{align}\label{def:j_lambda}
j_\lambda(x)=\mathbb{E}_0 \left[  \exp\left( -\int_0^\infty F( X^{\lambda}(t)-x_*+x)\, dt \right)\right]
\end{align}
is well-defined for all $x\in\mathbb{R}^N$ and 
\begin{align*}
\frac{1}{2}\Delta j_\lambda (x)+\left\langle \lambda,\nabla j_\lambda (x) \right\rangle= \left( \sum_{i=1}^{N-1} e^{x_{i+1}-x_i}\right) j_\lambda	(x).
\end{align*}
Moreover, $ j_\lambda(x)$ is proportional to $ e^{-\left\langle \lambda,x \right\rangle}\psi_{\lambda}(x)$. Let
\begin{align*}
M(t)=\frac{j_\lambda(X^\lambda(t))}{j_\lambda(X^\lambda(0))}\exp\left( -\int_0^t F(X^\lambda(s) )\, ds.\right).
\end{align*} 
It is straightforward to check that for $M(t)$ is an non-negative $\mathbb{P}_0$-martingale. From It{\^o}'s formula, $dM(t)=M(t)dL(t)$ with 
\begin{align*}
L(t)=\int_0^{t}
\sum_{i=1}^N \nabla_i\log j_\lambda (X^\lambda(s))  dX_i(s).
\end{align*}
Let $ {\mathbb{P}} $ be the measure obtained through tilting $\mathbb{P}_0$ by $M(t)$. From the Girsanov theorem \cite[Theorem 1.7, Chapter VIII] {MR1725357}, under $ {\mathbb{P}}$, $X^\lambda(t)$ is the solution to \eqref{equ:OClambda} which starts at $x_*$. 

The martingale property of $M(t)$ implies that for any $t_0>0,$ $d {\mathbb{P}}/d\mathbb{P}_0|_{\mathcal{F}_{t_0}}=M(t_0)$. Therefore, the $ {\mathbb{P}} $-marginal law of $X^\lambda(t)$ on $C(\llbracket 1,N \rrbracket\times [0,t_0] ,\mathbb{R})$, denoted by $\mu$, is of the form $\mu_{\mathsf{s},\mathbf{H}}$ with $\mathsf{s}=(\llbracket 1,N \rrbracket\times [0,t_0],x_*,\lambda t)$ and 
\begin{align*}
\mathbf{H}(z)=  j_\lambda(z(t_0)) \exp\left( -\int_0^{t_0} F(z(s)  )\, ds.\right)
\end{align*}
It remains to check $F(x)$ is convex and $j_\lambda(x)$ is log-concave. Clearly $F(x)=\sum_{i=1}^{N-1} e^{x_{i+1} -x_i }$ is a convex function. From Lemma~\ref{lem:logconvexmarginal},
$$
\mathbb{E}_0 \left[  \exp\left( -\int_0^T F( X^{\lambda}(t)-x_*+x)\, dt \right)\right]
$$
is log-concave in $x$ for any $T>0$. By taking $T\to\infty$, $ j_\lambda (x)$ is log-concave as well. Therefore, $\mu\in \mathring{\mathfrak{M}}$. Since this holds for all $t_0>0$, we conclude the ${\mathbb{P}}$-law of $X^\lambda(t)$ belongs to $\mathfrak{M}$ and $X^\lambda\in\mathbf{LC}$.
\end{proof}

\begin{proof}[Proof of Proposition~\ref{prop:OY}]
Let $x^{(k)}_*=\left(  \frac{k(-N+1)}{2}, \frac{k(-N+3)}{2},\dots, \frac{k(N-1)}{2} \right)$ and let $X^{(k)}(t)$ be the the solution to \eqref{equ:OClambda} that starts at $x_*^{(k)}$. It is shown in \cite[Section 8]{MR2952082} that $X^{(k)}$ converges to $Y^{\lambda}$ in distribution. From Lemma~\ref{lem:qtinm}, $X^{(k)}$ is in the class $\mathbf{LC}$. Since $\mathbf{LC}$ is closed under the convergence in distribution, $Y^{\lambda}$ is also in the class $\mathbf{LC}$.
\end{proof}

\bibliographystyle{alpha} 
\bibliography{CLE.bib}

\newcommand{\noopsort}[1]{} \newcommand{\singleletter}[1]{#1}
\begin{thebibliography}{AFvM10}

\bibitem[AFvM10]{MR2642890}
Mark Adler, Patrik~L. Ferrari, and Pierre van Moerbeke.
\newblock Airy processes with wanderers and new universality classes.
\newblock {\em Ann. Probab.}, 38(2):714--769, 2010.

\bibitem[AGZ10]{MR2760897}
Greg~W. Anderson, Alice Guionnet, and Ofer Zeitouni.
\newblock {\em An introduction to random matrices}, volume 118 of {\em
  Cambridge Studies in Advanced Mathematics}.
\newblock Cambridge University Press, Cambridge, 2010.

\bibitem[AH23]{aggarwal2023}
Amol Aggarwal and Jiaoyang Huang.
\newblock Strong characterization for the {A}iry line ensemble.
\newblock 2023.
\newblock arXiv:2308.11908.

\bibitem[Bil99]{Bil}
P.~Billingsley.
\newblock {\em Convergence of probability measures}.
\newblock Wiley Series in Probability and Statistics: Probability and
  Statistics. John Wiley \& Sons, Inc., New York, second edition, 1999.
\newblock A Wiley-Interscience Publication.

\bibitem[BKM24]{baslingker}
Jnaneshwar Baslingker, Manjunath Krishnapur, and Mokshay Madiman.
\newblock Log-concavity in one-dimensional coulomb gases and related ensembles.
\newblock 2024.
\newblock arXiv:2412.15116.

\bibitem[BO11]{MR2884226}
Fabrice Baudoin and Neil O'Connell.
\newblock Exponential functionals of {B}rownian motion and class-one
  {W}hittaker functions.
\newblock {\em Ann. Inst. Henri Poincar\'e{} Probab. Stat.}, 47(4):1096--1120,
  2011.

\bibitem[Caf00]{MR1800860}
Luis~A. Caffarelli.
\newblock Monotonicity properties of optimal transportation and the {FKG} and
  related inequalities.
\newblock {\em Comm. Math. Phys.}, 214(3):547--563, 2000.

\bibitem[Car03]{MR2004294}
John Cardy.
\newblock Stochastic {L}oewner evolution and {D}yson's circular ensembles.
\newblock {\em J. Phys. A}, 36(24):L379--L386, 2003.

\bibitem[CH14]{MR3152753}
Ivan Corwin and Alan Hammond.
\newblock Brownian {G}ibbs property for {A}iry line ensembles.
\newblock {\em Invent. Math.}, 195(2):441--508, 2014.

\bibitem[CH16]{CH16}
Ivan Corwin and Alan Hammond.
\newblock K{PZ} line ensemble.
\newblock {\em Probab. Theory Related Fields}, 166(1-2):67--185, 2016.

\bibitem[Dim24]{dimitrov2024}
Evgeni Dimitrov.
\newblock Airy wanderer line ensembles.
\newblock 2024.
\newblock arXiv:2408.08445.

\bibitem[DJD88]{MR954608}
Sudhakar Dharmadhikari and Kumar Joag-Dev.
\newblock {\em Unimodality, convexity, and applications}.
\newblock Probability and Mathematical Statistics. Academic Press, Inc.,
  Boston, MA, 1988.

\bibitem[DM21]{MR4317710}
Evgeni Dimitrov and Konstantin Matetski.
\newblock Characterization of {B}rownian {G}ibbsian line ensembles.
\newblock {\em Ann. Probab.}, 49(5):2477--2529, 2021.

\bibitem[DNV19]{dauvergne2019uniform}
Duncan Dauvergne, Mihai Nica, and B{\'a}lint Vir{\'a}g.
\newblock Uniform convergence to the airy line ensemble.
\newblock {\em arXiv preprint}, 2019.

\bibitem[DOV22]{Dauvergne_2022}
Duncan Dauvergne, Janosch Ortmann, and Bálint Virág.
\newblock The directed landscape.
\newblock {\em Acta Mathematica}, 229(2):201–285, 2022.

\bibitem[DV21]{MR4260467}
Duncan Dauvergne and B\'alint Vir\'ag.
\newblock Bulk properties of the {A}iry line ensemble.
\newblock {\em Ann. Probab.}, 49(4):1738--1777, 2021.

\bibitem[DV22]{dauvergne2022scaling}
Duncan Dauvergne and Bálint Virág.
\newblock The scaling limit of the longest increasing subsequence.
\newblock 2022.
\newblock arXiv:2104.08210.

\bibitem[Dys62]{MR148397}
Freeman~J. Dyson.
\newblock A {B}rownian-motion model for the eigenvalues of a random matrix.
\newblock {\em J. Mathematical Phys.}, 3:1191--1198, 1962.

\bibitem[FU04]{MR2036490}
D.~Feyel and A.~S. \"Ust\"unel.
\newblock Monge-{K}antorovitch measure transportation and {M}onge-{A}mp\`ere
  equation on {W}iener space.
\newblock {\em Probab. Theory Related Fields}, 128(3):347--385, 2004.

\bibitem[GRR71]{MR267632}
A.~M. Garsia, E.~Rodemich, and H.~Rumsey, Jr.
\newblock A real variable lemma and the continuity of paths of some {G}aussian
  processes.
\newblock {\em Indiana Univ. Math. J.}, 20:565--578, 1970/71.

\bibitem[GXZ24]{gorin2024}
Vadim Gorin, Jiaming Xu, and Lingfu Zhang.
\newblock Airy$_\beta$ line ensemble and its {L}aplace transform.
\newblock 2024.
\newblock arXiv:2411.10829.

\bibitem[Har04]{MR2095937}
Gilles Harg\'e.
\newblock A convex/log-concave correlation inequality for {G}aussian measure
  and an application to abstract {W}iener spaces.
\newblock {\em Probab. Theory Related Fields}, 130(3):415--440, 2004.

\bibitem[HZ24]{huang2024convergenceframeworkairybetaline}
Jiaoyang Huang and Lingfu Zhang.
\newblock A convergence framework for {A}iry$_\beta$ line ensemble via pole
  evolution.
\newblock 2024.
\newblock arXiv:2411.10586.

\bibitem[Jan97]{MR1474726}
Svante Janson.
\newblock {\em Gaussian {H}ilbert spaces}, volume 129 of {\em Cambridge Tracts
  in Mathematics}.
\newblock Cambridge University Press, Cambridge, 1997.

\bibitem[KPZ86]{KPZ}
Mehran Kardar, Giorgio Parisi, and Yi-Cheng Zhang.
\newblock Dynamic scaling of growing interfaces.
\newblock {\em Phys. Rev. Lett.}, 56:889--892, 3 1986.

\bibitem[Lan20]{landon2020edgescalinglimitdyson}
Benjamin Landon.
\newblock Edge scaling limit of {D}yson {B}rownian motion at equilibrium for
  general $\beta \geq 1$.
\newblock 2020.
\newblock arXiv:2009.11176.

\bibitem[MS24]{MR4797372}
Dan Mikulincer and Yair Shenfeld.
\newblock The {B}rownian transport map.
\newblock {\em Probab. Theory Related Fields}, 190(1-2):379--444, 2024.

\bibitem[Nic21]{Nic21}
Mihai Nica.
\newblock Intermediate disorder limits for multi-layer semi-discrete directed
  polymers.
\newblock {\em Electron. J. Probab.}, 26:Paper No. 62, 50, 2021.

\bibitem[O'C12]{MR2952082}
Neil O'Connell.
\newblock Directed polymers and the quantum {T}oda lattice.
\newblock {\em Ann. Probab.}, 40(2):437--458, 2012.

\bibitem[OW16]{MR3439221}
Neil O'Connell and Jon Warren.
\newblock A multi-layer extension of the stochastic heat equation.
\newblock {\em Comm. Math. Phys.}, 341(1):1--33, 2016.

\bibitem[OY01]{MR1865759}
Neil O'Connell and Marc Yor.
\newblock Brownian analogues of {B}urke's theorem.
\newblock {\em Stochastic Process. Appl.}, 96(2):285--304, 2001.

\bibitem[PS02]{MR1933446}
Michael Pr\"ahofer and Herbert Spohn.
\newblock Scale invariance of the {PNG} droplet and the {A}iry process.
\newblock volume 108, pages 1071--1106. 2002.
\newblock Dedicated to David Ruelle and Yasha Sinai on the occasion of their
  65th birthdays.

\bibitem[Qua14]{MR3204148}
J.~D. Quastel.
\newblock The {K}ardar-{P}arisi-{Z}hang equation and universality class.
\newblock In {\em X{VII}th {I}nternational {C}ongress on {M}athematical
  {P}hysics}, pages 113--133. World Sci. Publ., Hackensack, NJ, 2014.

\bibitem[RRV11]{MR2813333}
Jos\'e{}~A. Ram\'irez, Brian Rider, and B\'alint Vir\'ag.
\newblock Beta ensembles, stochastic {A}iry spectrum, and a diffusion.
\newblock {\em J. Amer. Math. Soc.}, 24(4):919--944, 2011.

\bibitem[RS93]{MR1217451}
L.~C.~G. Rogers and Z.~Shi.
\newblock Interacting {B}rownian particles and the {W}igner law.
\newblock {\em Probab. Theory Related Fields}, 95(4):555--570, 1993.

\bibitem[RY99]{MR1725357}
Daniel Revuz and Marc Yor.
\newblock {\em Continuous martingales and {B}rownian motion}, volume 293 of
  {\em Grundlehren der mathematischen Wissenschaften [Fundamental Principles of
  Mathematical Sciences]}.
\newblock Springer-Verlag, Berlin, third edition, 1999.

\bibitem[Sim11]{MR2814377}
Barry Simon.
\newblock {\em Convexity}, volume 187 of {\em Cambridge Tracts in Mathematics}.
\newblock Cambridge University Press, Cambridge, 2011.
\newblock An analytic viewpoint.

\bibitem[SV06]{MR2190038}
Daniel~W. Stroock and S.~R.~Srinivasa Varadhan.
\newblock {\em Multidimensional diffusion processes}.
\newblock Classics in Mathematics. Springer-Verlag, Berlin, 2006.
\newblock Reprint of the 1997 edition.

\bibitem[Tsa16]{MR3568040}
Li-Cheng Tsai.
\newblock Infinite dimensional stochastic differential equations for {D}yson's
  model.
\newblock {\em Probab. Theory Related Fields}, 166(3-4):801--850, 2016.

\bibitem[TW94]{MR1257246}
Craig~A. Tracy and Harold Widom.
\newblock Level-spacing distributions and the {A}iry kernel.
\newblock {\em Comm. Math. Phys.}, 159(1):151--174, 1994.

\bibitem[Wu22]{Wu21tightness}
Xuan Wu.
\newblock Convergence of the {KPZ} line ensemble.
\newblock {\em International Mathematics Research Notices}, 2022.

\bibitem[Wu23a]{MR4605337}
Xuan Wu.
\newblock The {B}essel line ensemble.
\newblock {\em Electron. J. Probab.}, 28:Paper No. 77, 50, 2023.

\bibitem[Wu23b]{wu2023kpzequationdirectedlandscape}
Xuan Wu.
\newblock The {KPZ} equation and the directed landscape.
\newblock 2023.
\newblock arXiv:2301.00547.

\end{thebibliography}

\end{document}